\documentclass[11pt,a4paper]{article}

\usepackage[utf8]{inputenc}
\usepackage[T1]{fontenc}
\usepackage{lmodern}
\usepackage[english]{babel}
\usepackage{microtype}

\usepackage{mathtools}
\usepackage{amssymb}
\usepackage{amsthm}

\usepackage[
left=1.08in,
right=1.08in,
top=1.10in,
bottom=1.10in
]{geometry}

\allowdisplaybreaks[3]
\numberwithin{equation}{section}

\theoremstyle{plain}

\newtheorem{theorem}{Theorem}[section]
\newtheorem{proposition}[theorem]{Proposition}
\newtheorem{lemma}[theorem]{Lemma}

\theoremstyle{definition}

\theoremstyle{remark}

\DeclareMathOperator{\End}{End}
\DeclareMathOperator{\Tr}{Tr}
\newcommand{\Id}{\operatorname{Id}}

\usepackage[hidelinks]{hyperref}

\title{Limiting Behavior of a Class of Hermitian Yang--Mills Metrics, II:
Exponential Approximation}
\author{Jixiang Fu, Dekai Zhang}
\date{}

\begin{document}

\maketitle

\begin{abstract}
This paper is a sequel to \cite{Fu}, where the first-named author
constructed a family of approximate Hermitian Yang--Mills metrics
\(H_{0,\epsilon}\) on stable rank-two holomorphic vector bundles arising
from double spectral covers over the product of two one-dimensional
complex tori.

We prove that these approximate metrics give an all-order,
exponentially accurate asymptotic description of the exact Hermitian
Yang--Mills metrics in the large K\"ahler limit. More precisely, the
mean curvature of \(H_{0,\epsilon}\) decays exponentially in every
\(C^k\)-norm. Moreover, if \(H_{1,\epsilon}\) denotes the exact
Hermitian Yang--Mills metric and
\[
H_\epsilon=H_{0,\epsilon}^{-1}H_{1,\epsilon},
\]
then, after normalization, for every nonnegative integer \(k\), there
exist positive constants \(C_k\) and \(c_k\) such that
\[
\|H_\epsilon-\Id\|_{C^k}\leq C_k e^{-c_k/\epsilon}.
\]
The main analytic difficulty lies in the global \(C^0\)-comparison. Obtaining \(C^0\)-estimates for the coupled nonlinear Hermitian Yang--Mills system is intrinsically difficult; moreover, the equation controls only the contraction of the curvature, and hence only certain combinations of second derivatives, whereas one needs global control of the full matrix-valued metric.
\end{abstract}

\section{Introduction}
The complex Monge--Amp\`ere equation is a fundamental geometric PDE in K\"ahler geometry. Yau \cite{Yau} solved the Calabi conjecture and, in particular, established the existence of Ricci-flat K\"ahler metrics, namely Calabi--Yau metrics, under the appropriate topological assumption. The Hermitian Yang--Mills equation is another fundamental geometric PDE in K\"ahler geometry. Donaldson \cite{Don1985,Don1987} and Uhlenbeck--Yau \cite{UY} established the existence of Hermitian Yang--Mills metrics on stable holomorphic vector bundles.

Existence alone, however, generally gives little explicit information about the resulting canonical metrics. A fundamental example is the work of Gross and Wilson \cite{GW} on Calabi--Yau metrics on elliptically fibered K3 surfaces near a large complex structure limit. They constructed approximate Ricci-flat K\"ahler metrics adapted to the degeneration and proved that the exact Calabi--Yau metrics are exponentially close to them in all orders. Their work gives a precise quantitative description of the asymptotic geometry that is not visible from the existence theorem alone, and motivated many subsequent investigations of related degeneration problems
\cite{Wilson2004,Zharkov2004,Tosatti2010,RuanZhang2011,GTZ2013,ChenViaclovskyZhang2020,HeinTosatti2020,HeinTosatti2025}.

This paper is a sequel to \cite{Fu}, where the first-named author initiated an analogous study for Hermitian Yang--Mills metrics in a large K\"ahler limit and constructed a family of approximate Hermitian Yang--Mills metrics from the geometry of a double spectral cover.

More precisely, let \(X=B\times T\) be the product of two one-dimensional complex tori, and let \(\omega_\epsilon\) be the family of product K\"ahler metrics whose restrictions to \(B\) and \(T\) have areas \(\epsilon^{-1}\) and \(\epsilon\), respectively. Thus, as \(\epsilon\to0\), the base direction is stretched while the fiber direction collapses. Let \(V\) be a rank-two holomorphic vector bundle over \(X\) constructed from a double spectral cover \cite{Fri,FMW}, with
\[
c_1(V)=0.
\]
For all sufficiently small \(\epsilon\), the bundle \(V\) is slope-stable with respect to \(\omega_\epsilon\). Hence the Donaldson--Uhlenbeck--Yau theorem gives a Hermitian Yang--Mills metric \(H_{1,\epsilon}\) satisfying
\[
\Lambda_{\omega_\epsilon}\Theta(H_{1,\epsilon})=0.
\]

Let \(H_{0,\epsilon}\) be the approximate Hermitian Yang--Mills metric constructed in \cite{Fu}. To compare the two metrics, define the relative endomorphism
\[
H_\epsilon=H_{0,\epsilon}^{-1}H_{1,\epsilon}.
\]
It is positive definite and self-adjoint with respect to \(H_{0,\epsilon}\). After an appropriate normalization of \(H_{1,\epsilon}\), we may assume
\[
\det H_\epsilon=1.
\]
Our main result shows that the approximate metrics \(H_{0,\epsilon}\) describe the exact Hermitian Yang--Mills metrics with exponential accuracy in all orders.

\begin{theorem}\label{Thm1.1}
For every nonnegative integer \(k\), there exist positive constants \(C_k\) and \(c_k\) such that, for all sufficiently small \(\epsilon>0\),
\[
\|H_\epsilon-\Id\|_{C^k}\leq C_k e^{-c_k/\epsilon}.
\]
\end{theorem}

Equivalently, in every fixed \(C^k\)-norm,
\[
H_{0,\epsilon}^{-1}H_{1,\epsilon}
=
\Id+O_{C^k}(e^{-c_k/\epsilon}).
\]
Thus the approximate metrics constructed in \cite{Fu} give an all-order, exponentially accurate asymptotic model for the exact Hermitian Yang--Mills metrics.

We now explain the geometry behind the approximation and the main difficulties in proving the theorem. The construction of \(H_{0,\epsilon}\) in \cite{Fu} is dictated by the spectral cover. Away from the branch locus, the limiting metric is explicitly determined by the spectral data. Near a branch point, this limiting metric becomes singular and is replaced by a smooth local Hermitian Yang--Mills model governed by a nonlinear radial equation. These local models are glued to the limiting metric across fixed annular regions and then conformally normalized to produce a globally smooth Hermitian metric on \(V\). The normalization removes the trace part of the mean curvature, so that
\[
\operatorname{Tr}\bigl(\Lambda_{\omega_\epsilon}\Theta(H_{0,\epsilon})\bigr)=0.
\]
In this sense, \(H_{0,\epsilon}\) is a geometrically constructed candidate for the asymptotic behavior of the exact Hermitian Yang--Mills metric.

The main analytic difficulty is the global \(C^0\)-comparison. Obtaining \(C^0\)-estimates for the coupled nonlinear Hermitian Yang--Mills system is intrinsically difficult. Moreover, the equation controls only the contraction of the curvature, and therefore only certain contracted combinations of second derivatives of the metric, together with nonlinear first-order terms, whereas the desired estimate requires global control of the full matrix-valued metric. The degeneration of the background K\"ahler metrics \(\omega_\epsilon\) introduces a further difficulty, since the constants in the usual Sobolev and elliptic estimates need not remain uniform as \(\epsilon\to0\).

We first establish exponential decay of the mean curvature of the approximate Hermitian Yang--Mills metrics, which is a key ingredient in the proof of the comparison theorem.

\begin{theorem}\label{Thm1.2}
For every nonnegative integer \(k\), there exist positive constants \(C_k\) and \(c_k\) such that, for all sufficiently small \(\epsilon>0\),
\[
\left\|
\Lambda_{\omega_\epsilon}\Theta(H_{0,\epsilon})
\right\|_{C^k}
\leq C_k e^{-c_k/\epsilon}.
\]
\end{theorem}

The exponential rate originates in the local model near the branch points. After rescaling, the difference between the smooth radial solution and the singular limiting solution satisfies a singularly perturbed second-order equation with a large positive zeroth-order term. A comparison argument gives exponential decay on the fixed gluing annulus, and the equation then yields exponential estimates for all higher derivatives. Since the mean curvature of \(H_{0,\epsilon}\) is supported in the gluing regions, this gives the global estimate above.

Passing from exponentially small mean curvature to exponential closeness of the metrics is a genuinely global problem. Since \(\det H_\epsilon=1\) and \(V\) has rank two, one has
\[
\Tr H_\epsilon\geq2.
\]
Thus the \(C^0\)-comparison is essentially reduced to controlling \(\sup_X(\Tr H_\epsilon-2)\). The Hermitian Yang--Mills equation gives a differential inequality for \(\Tr H_\epsilon\). A Moser iteration based on a Sobolev inequality adapted to the degenerating metrics controls its \(L^\infty\)-norm in terms of its averaged part. To control the latter, we compare \(H_{0,\epsilon}\) with a fixed Hermitian metric and use the simplicity of the stable bundle to obtain a fixed Poincar\'e-type inequality on \(\End(V)\). Together with an energy estimate and the determinant normalization, this yields
\[
\sup_X\left|\Tr H_\epsilon-2\right|
\leq e^{-c/\epsilon},
\]
and hence the exponential \(C^0\)-comparison. The higher-order estimates then follow by elliptic arguments; the polynomial losses caused by the degenerating geometry are absorbed by the exponential decay.

Viewed together, \cite{Fu} and the present paper address the two basic steps of the approximation problem. The former constructs the approximate Hermitian Yang--Mills metrics from the spectral-cover geometry, while the present paper proves that these metrics capture the exact solutions with exponential accuracy. In this sense, the two papers provide, in the present setting, a Hermitian Yang--Mills analogue of the exponential approximation picture of Gross and Wilson.

Related convergence results on collapsing elliptically fibered K3 surfaces were obtained by Datar--Jacob \cite{DatarJacob2022} and Datar--Jacob--Zhang \cite{DatarJacobZhang2021}. Their work treats the more general non-flat K3 geometry and establishes convergence on generic fibers and subsequential convergence away from finitely many fibers, respectively. The result here is of a different, quantitative nature: in the flat product-torus setting, the spectral-cover construction provides a global approximate metric, and the exact Hermitian Yang--Mills metric is shown to be exponentially close to it in every order. The flat product-torus setting was originally introduced in \cite{Fu} as a model problem for the corresponding question on elliptically fibered K3 surfaces, with the aim of first understanding the local geometry and the Hermitian Yang--Mills model near the branch points of the spectral cover before passing to the non-flat K3 setting.

The paper is organized as follows. In Section~2 we recall the spectral-cover geometry and the construction of the approximate Hermitian Yang--Mills metrics in \cite{Fu}. In Section~3 we establish the exponential estimates for the local radial model and deduce the exponential decay of the mean curvature of \(H_{0,\epsilon}\), thereby proving Theorem~\ref{Thm1.2}. In Section~4 we prove the global exponential \(C^0\)-estimate for \(H_\epsilon\) and derive the higher-order estimates,  completing the proof of Theorem~\ref{Thm1.1}.

\section{Preliminaries}
\label{sec:preliminaries}

In this section, we recall the geometric setting and the construction
of the approximate Hermitian Yang--Mills metrics in \cite{Fu}.

Following \cite[Sections~1--2]{Fu}, let
\(\Gamma=\mathbb Z+\sqrt{-1}\mathbb Z\), and let
\(\mathbb C^\ast\) denote the dual complex vector space of
\(\mathbb C\), with dual lattice \(\Gamma^\ast\). Let \(B\) and \(T\)
be two copies of the one-dimensional complex torus
\(\mathbb C/\Gamma\), and set \(X=B\times T\). Let
\(T^\ast=\mathbb C^\ast/\Gamma^\ast\) be the dual torus of \(T\), and
set \(X^\ast=B\times T^\ast\). We write
\[
z=x_1+\sqrt{-1}x_2,\qquad
w=y_1+\sqrt{-1}y_2,\qquad
w^\ast=y_1^\ast+\sqrt{-1}y_2^\ast
\]
for the standard complex coordinates on the respective universal
covers of \(B\), \(T\), and \(T^\ast\). We equip \(X\) with the family
of K\"ahler metrics 
\[
\omega_\epsilon
=
\frac{\sqrt{-1}}{2}\epsilon^{-1}dz\wedge d\bar z
+
\frac{\sqrt{-1}}{2}\epsilon\,dw\wedge d\bar w.
\]
 Then the volumes $\frac{\omega_\epsilon^2}{2}$ is independent of \(\epsilon\).

Let \(Y\subset X^\ast\) be a smooth spectral curve. Denote by
\(\varphi:Y\to B\) and \(q:Y\to T^\ast\) the restrictions of the two
projections, and suppose that \(\varphi\) is a double cover with branch
divisor
\[
D_0=\sum_{a=1}^n\xi_a.
\]
We assume that \(4\mid n\), choose distinct points
\(\xi_{n+1},\ldots,\xi_{5n/4}\) away from
\(\operatorname{Supp} D_0\), and set
\[
D_1=\sum_{j=n+1}^{5n/4}\xi_j.
\]
As in \cite{Fu}, we define the rank two holomorphic vector bundle
\[
V
=
p_{2*}\bigl(\iota^\ast\mathcal P\otimes p_1^\ast \varphi^\ast\mathcal O_B(D_1)\bigr)
\longrightarrow X,
\]
where \(\mathcal P\to T^\ast\times T\) is the Poincar\'e line bundle,
\[
\iota=(q,\operatorname{id}_T):Y\times T\longrightarrow T^\ast\times T,
\qquad
p_1=\operatorname{pr}_Y:Y\times T\longrightarrow Y,
\]
and
\[
p_2=(\varphi,\operatorname{id}_T):Y\times T\longrightarrow X.
\]
This construction is based on the spectral-cover constructions of
Friedman \cite{Fri} and Friedman--Morgan--Witten \cite{FMW}, and it
satisfies
\[
c_1(V)=0.
\]

By the adiabatic argument in \cite{Fu}, following \cite{FMW}, the
bundle \(V\) is slope-stable with respect to \(\omega_\epsilon\) for
all sufficiently small \(\epsilon\). Hence, for every sufficiently
small \(\epsilon\), there exists a Hermitian Yang--Mills metric
\(H_{1,\epsilon}\) on \(V\) whose Chern connection is irreducible.
Since \(c_1(V)=0\), its mean curvature vanishes:
\[
\Lambda_{\omega_\epsilon}\Theta(H_{1,\epsilon})=0,
\]
where \(\Lambda_{\omega_\epsilon}\) denotes contraction with
\(\omega_\epsilon\).

We next recall the approximate metric constructed in
\cite[Sections~2--5]{Fu}. Choose \(r_0>0\) sufficiently small so that
the disks of radius \(2r_0\) centered at
\(\xi_1,\ldots,\xi_{5n/4}\) are pairwise disjoint. Set
\[
D=D_0-4D_1,
\]
which is a divisor of degree zero on \(B\), and let \(G\) be the Green
function associated with \(D\). Choose a local coordinate \(z_\alpha\)
centered at each \(\xi_\alpha\). Near a branch point \(\xi_a\), one has
\[
G(z_a)=-\log|z_a|+2g_a(z_a),
\]
where \(g_a\) is real-valued and harmonic. Near a point
\(\xi_j\in\operatorname{Supp}(D_1)\), one similarly has
\[
G(z_j)=4\log|z_j|+2g_j(z_j),
\]
where \(g_j\) is real-valued and harmonic.

With respect to the smooth frame used in \cite{Fu} over
\[
\bigl(B\setminus
(\operatorname{Supp} D_0\cup\operatorname{Supp} D_1)\bigr)\times T,
\]
define
\[
h_0=e^{G/2}\Id.
\]
Then \(h_0\) is a Hermitian Yang--Mills metric on this complement.
Using the corresponding transition functions, it extends smoothly
across \(D_1\), but is singular along \(D_0\).

Near each branch point \(\xi_a\), after translating the local
\(T^\ast\)-coordinate if necessary, choose a coordinate \(z_a\) on
\(B\) such that the spectral cover is locally given by
\[
(w^\ast)^2=z_a.
\]
Following \cite[Sections~3--4]{Fu}, let
\[
h_{a,\epsilon}
=
e^{g_a(z_a)}
\begin{pmatrix}
e^{-u_\epsilon}&0\\
0&e^{u_\epsilon}
\end{pmatrix}
\]
be the corresponding smooth local Hermitian Yang--Mills metric, where
\(u_\epsilon\) is the unique smooth radial solution of the Dirichlet
problem recalled in the next section.

Choose a fixed smooth cutoff function
\(\rho=\rho(|z_a|^2)\) such that
\[
\rho=1\quad\text{for }|z_a|\leq r_0,
\qquad
\rho=0\quad\text{for }|z_a|\geq\frac{4r_0}{3}.
\]
On each branch disk, define
\[
h_\epsilon
=
(1-\rho)h_0+\rho h_{a,\epsilon},
\]
and set \(h_\epsilon=h_0\) elsewhere. This defines a globally smooth
Hermitian metric on \(V\).

On each branch disk, write
\[
h_\epsilon
=
e^{g_a(z_a)}
\begin{pmatrix}
e^{\phi_1}&0\\
0&e^{\phi_2}
\end{pmatrix},
\]
where
\[
e^{\phi_1}
=
(1-\rho)|z_a|^{-1/2}+\rho e^{-u_\epsilon},
\qquad
e^{\phi_2}
=
(1-\rho)|z_a|^{1/2}+\rho e^{u_\epsilon}.
\]
Since \(\phi_1+\phi_2\) vanishes near the boundary of each branch
disk, it extends smoothly by zero outside the branch disks. The
approximate metric is then defined by
\[
H_{0,\epsilon}
=
e^{-\frac12(\phi_1+\phi_2)}h_\epsilon.
\]
This conformal normalization gives
\[
\Tr\bigl(
\Lambda_{\omega_\epsilon}\Theta(H_{0,\epsilon})
\bigr)=0.
\]
Moreover, the mean curvature
\(\Lambda_{\omega_\epsilon}\Theta(H_{0,\epsilon})\) is supported in
the union of the fixed gluing annuli
\[
\bigcup_{a=1}^n
\Bigl\{
r_0\leq |z_a|\leq \frac{4r_0}{3}
\Bigr\}\times T.
\]

\section{Exponential Decay of the Mean Curvature of the Approximate HYM Metric}
\label{sec:mean-curvature}

We now establish exponential estimates for the local radial solution on a fixed annulus and use them to prove exponential decay of the mean curvature of $H_{0,\epsilon}$. The required properties of the radial solution are contained in \cite[Theorem~6 and Proposition~7]{Fu}.

Let $\xi_a$ be a branch point and let $D_R=\{z_a:|z_a|<R\}$, where $R=2r_0$ is independent of $\epsilon$. Write $r=|z_a|$. By \cite[Section~3 and Theorem~6]{Fu}, $u_\epsilon$ is the unique smooth solution, independent of the fiber variable, of
\[
\begin{cases}
\displaystyle \Delta u_\epsilon=\dfrac{4\pi^2}{\epsilon^2}\left(e^{2u_\epsilon}-|z_a|^2e^{-2u_\epsilon}\right),&\text{in }D_R,\\[4pt]
\displaystyle u_\epsilon=\dfrac12\log R,&\text{on }\partial D_R.
\end{cases}
\]
Here $\Delta$ denotes the Euclidean Laplacian in the $z_a$-coordinate. Since the equation and the boundary condition are rotationally invariant, uniqueness implies that $u_\epsilon$ is radial. Thus $u_\epsilon=u_\epsilon(r)$ and
\begin{equation}
	\left\{
	\begin{aligned}
		u_\epsilon''(r)+\frac{1}{r}u_\epsilon'(r)
		&=\frac{4\pi^2}{\epsilon^2}
		\left(e^{2u_\epsilon(r)}-r^2e^{-2u_\epsilon(r)}\right),
		&& 0<r<R,\\
		u_\epsilon(R)
		&=\frac{1}{2}\log R.
	\end{aligned}
	\right.
	\label{eq:fu-ode}
\end{equation}
Moreover, smoothness at the origin gives $u_\epsilon'(0)=0$. The boundary value agrees with that of the singular solution $u_0(r)=\frac12\log r$.

To compare with the normalized equation in \cite[Proposition~7]{Fu}, set
\begin{equation}
s=\frac{r}{2r_0},\qquad
\varepsilon=\frac{\epsilon}{8\pi r_0^{3/2}},\qquad
\overline u_{\varepsilon}(s)=2u_\epsilon(2r_0s)-\log(2r_0).
\label{eq:radial-rescaling}
\end{equation}
A direct calculation gives
\[
\overline u_{\varepsilon}''+s^{-1}\overline u_{\varepsilon}'
=\varepsilon^{-2}\left(e^{\overline u_{\varepsilon}}-s^2e^{-\overline u_{\varepsilon}}\right),
\qquad 0<s<1,\quad \overline u_{\varepsilon}(1)=0,
\]
which is equation~(4.2) in \cite{Fu}. Since $\overline u_{\varepsilon}'(s)=4r_0u_\epsilon'(2r_0s)$, Proposition~7 of \cite{Fu} implies that for $0\leq r\leq R$, we have
\[
u_\epsilon'(r)\geq0.
\]

The following proposition gives exponential estimates for the radial solution on the fixed annulus $[r_0,2r_0]$.

\begin{proposition}
\label{prop:exponential-ode}
For any nonnegative integer $k$, there exist constants $C_k,c_k>0$, independent of $\epsilon$, such that for all sufficiently small $\epsilon>0$,
\begin{equation}
\left\|u_\epsilon-\frac12\log r\right\|_{C^k([r_0,2r_0])}\leq C_ke^{-c_k/\epsilon}.
\label{eq:exponential-ode}
\end{equation}
\end{proposition}

\begin{proof}
Use the rescaled variables in \eqref{eq:radial-rescaling}, and continue to regard $\varepsilon$ as the small parameter. Set
\[
v_{\varepsilon}(s)=\overline u_{\varepsilon}(s)-\log s.
\]
Then
\begin{equation}
v_{\varepsilon}''+s^{-1}v_{\varepsilon}'
=2\varepsilon^{-2}s\sinh v_{\varepsilon},
\qquad v_{\varepsilon}(1)=0.
\label{eq:v-ode}
\end{equation}
Lemma~8 of \cite{Fu} gives $v_{\varepsilon}>0$ and
$v_{\varepsilon}'<0$ on $(0,1)$; equation~\eqref{eq:v-ode} then gives
$v_{\varepsilon}''>0$. Since $\overline u_{\varepsilon}$ is smooth
and radial at the origin,
$\lim_{s\to0}s v_{\varepsilon}'(s)=-1$.

Multiplying \eqref{eq:v-ode} by $s$, integrating over $[\delta,1]$,
and then letting $\delta\downarrow0$, we obtain
\[
2\varepsilon^{-2}\int_0^1s^2\sinh v_{\varepsilon}(s)\,ds
=v_{\varepsilon}'(1)+1<1.
\]
Let $a=1/4$.  Since $v_{\varepsilon}$ is decreasing,
\[
\frac{a^3}{3}\sinh v_{\varepsilon}(a)
\leq\int_0^a s^2\sinh v_{\varepsilon}(s)\,ds
\leq\frac{\varepsilon^2}{2}.
\]
Since
$0<v_{\varepsilon}(a)\leq\sinh v_{\varepsilon}(a)$, we get
\[
v_{\varepsilon}(a)\leq C\varepsilon^2.
\]
For \(0<s<1\), define
\[
q_{\varepsilon}(s)
=
2\varepsilon^{-2}s
\frac{\sinh v_{\varepsilon}(s)}{v_{\varepsilon}(s)}.
\]
Since \(v_{\varepsilon}(1)=0\) and
\(\frac{\sinh t}{t}\to1\) as \(t\to0\), the function
\(q_{\varepsilon}\) extends continuously to \(s=1\) by setting
$
q_{\varepsilon}(1)=2\varepsilon^{-2}.
$
Moreover, 
\[
q_{\varepsilon}(s)
\geq 2\varepsilon^{-2}s
\geq \frac12\varepsilon^{-2},
\quad s\in [a, 1].
\]
Let \(L_{\varepsilon}\) be the second-order linear differential
operator on \([a,1]\) defined by
\[
L_{\varepsilon}f
:=
f''+\frac1s f'-q_{\varepsilon}f.
\]

Set $\mu=\frac{1}{2\varepsilon}$, and define
\[
\Phi_{\varepsilon}(s)=v_{\varepsilon}(a)
\frac{\sinh\bigl(\mu(1-s)\bigr)}{\sinh\bigl(\mu(1-a)\bigr)}.
\]
Then $\Phi_{\varepsilon}(a)=v_{\varepsilon}(a)$,
$\Phi_{\varepsilon}(1)=0$, $\Phi_{\varepsilon}'<0$, and
$\Phi_{\varepsilon}''=\mu^2\Phi_{\varepsilon}$. Thus
\[
L_{\varepsilon}\Phi_{\varepsilon}
=(\mu^2-q_{\varepsilon})\Phi_{\varepsilon}
+s^{-1}\Phi_{\varepsilon}'<0.
\]
Applying the maximum principle to $v_{\varepsilon}-\Phi_{\varepsilon}$, we have
\[
0\leq v_{\varepsilon}\leq\Phi_{\varepsilon}\quad\text{on}\quad {[a,1]}.
\]
Since
\[
\frac{\sinh(\mu(1-s))}{\sinh(\mu(1-a))}\leq Ce^{-\mu(s-a)},
\]
it follows that
\[
0\leq v_{\varepsilon}(s)\leq C\varepsilon^2e^{-\frac{1}{8\varepsilon}}
\leq Ce^{-\frac{1}{16\varepsilon}},\quad  s\in[\frac12, 1].
\]
The preceding comparison also gives
\[
\|v_{\varepsilon}\|_{C^0([3/8,1])}
\leq Ce^{-c/\varepsilon}.
\]

We now estimate the derivatives. Since $v_{\varepsilon}$ is decreasing
and convex, for $s\in[1/2,1]$ and $h=1/8$,
\[
0\leq-v_{\varepsilon}'(s)
\leq\frac{v_{\varepsilon}(s-h)-v_{\varepsilon}(s)}{h}
\leq8v_{\varepsilon}(s-h)
\leq Ce^{-c/\varepsilon}.
\]
Using \eqref{eq:v-ode} and $\sinh t\leq Ct$ for the exponentially
small values of $t=v_{\varepsilon}$ on this interval, we get
\[
\|v_{\varepsilon}''\|_{C^0([1/2,1])}
\leq C\varepsilon^{-2}e^{-c/\varepsilon}
\leq Ce^{-c_2/\varepsilon}.
\]

For the higher derivatives, rewrite \eqref{eq:v-ode} as
\[
v_{\varepsilon}''=-s^{-1}v_{\varepsilon}'
+2\varepsilon^{-2}s\sinh v_{\varepsilon}.
\]
After differentiating this identity $m$ times, any term coming from
$s^{-1}v_{\varepsilon}'$ contains a derivative of $v_{\varepsilon}$
of order between $1$ and $m+1$, while any term coming from
$s\sinh v_{\varepsilon}$ contains either $\sinh v_{\varepsilon}$ or
at least one derivative of $v_{\varepsilon}$ of order at most $m$.
Since all remaining factors are uniformly bounded on $[1/2,1]$,
induction, together with
$\varepsilon^{-N}e^{-c/\varepsilon}\leq
Ce^{-c'/\varepsilon}$, yields
\[
\|v_{\varepsilon}^{(m)}\|_{C^0([1/2,1])}\leq C_m e^{-\frac{c_m}{\varepsilon}},
\]
for any nonnegative integer $m$. Recall
$v_{\varepsilon}(s)=2\bigl(u_\epsilon(r)-\frac12\log r\bigr)$ and $\varepsilon$ is a fixed positive multiple of $\epsilon$, so we get \eqref{eq:exponential-ode}.
\end{proof}

Based on the above estimate, we can prove the following exponential estimate for the mean curvature of the approximate metric.
\begin{theorem}
\label{thm:approximate-mean-curvature}
For any nonnegative integer $k$, there exist constants $C_k,c_k>0$, independent of $\epsilon$, such that for all sufficiently small $\epsilon>0$,
\begin{equation}
\left\|\Lambda_{\omega_\epsilon}\Theta(H_{0,\epsilon})\right\|_{C^k}\leq C_ke^{-c_k/\epsilon}.
\label{eq:mean-curvature-exponential}
\end{equation}

\end{theorem}

\begin{proof}
	On a gluing annulus, set $\eta_\epsilon(r)=u_\epsilon(r)-\frac12\log r$. With respect to the local frame used in \cite[Section~5]{Fu}, let $\widetilde\Theta_{0,\epsilon}$ denote the matrix of $\Theta(H_{0,\epsilon})$. The curvature computation there gives
	\[
	\frac{\sqrt{-1}}{2}\Lambda_{\omega_\epsilon}\widetilde\Theta_{0,\epsilon}
	=\psi_\epsilon\begin{pmatrix}1&0\\0&-1\end{pmatrix},
	\]
	where \[
	\psi_\epsilon=\frac{\pi^2}{\epsilon}r\left(\varphi_\epsilon-\varphi_\epsilon^{-1}\right)
	-\frac{\epsilon}{2}\frac{\partial^2}{\partial z_a\partial\bar z_a}\log\varphi_\epsilon,
	\]
	and
	\[
	\varphi_\epsilon=\frac{1+\rho(r^2)\left(e^{-\eta_\epsilon}-1\right)}
	{1+\rho(r^2)\left(e^{\eta_\epsilon}-1\right)}.
	\]
	By Proposition~\ref{prop:exponential-ode},we have
	\[
	\|\eta_\epsilon\|_{C^{k+2}([r_0,2r_0])}\leq C_ke^{-c_k/\epsilon}.
	\]
	Since the gluing annulus lies in $r\geq r_0>0$, these radial estimates imply the corresponding coordinate $C^{k+2}$ estimates. The denominator defining $\varphi_\epsilon$ is therefore bounded below by a positive constant independent of $\epsilon$, and
	$\|\varphi_\epsilon-1\|_{C^{k+2}}\leq C_ke^{-c_k/\epsilon}$.
	
	Substituting into the formula for $\psi_\epsilon$, together with $\epsilon^{-N}e^{-c/\epsilon}\leq Ce^{-c'/\epsilon}$, we get
	\[
	\|\psi_\epsilon\|_{C^k}\leq C_ke^{-c_k/\epsilon}.
	\]
	Outside the gluing annuli, $\Lambda_{\omega_\epsilon}\Theta(H_{0,\epsilon})=0$. Hence the preceding local formula and the $C^k$-norm convention in \cite[Remark~4]{Fu} yield \eqref{eq:mean-curvature-exponential}.
\end{proof}
\section{The Global $C^0$ Estimate and Higher-Order Estimates}
\label{sec:trace-cn}

We now establish the exponential    $C^0$ decay estimate for $H_\epsilon-\Id$.  We
first derive a $C^0$ bound for the local radial solution $u_{\epsilon}$ and then
use it to compare the approximate metric $H_{0,\epsilon}$ with a fixed
  metric $H_{0,\epsilon_1}$.

\begin{lemma}
\label{lem:u-bound-cn}
There exists a constant $C>1$, independent of $\epsilon$, such that for
all sufficiently small $\epsilon>0$,
\[
C^{-1}\epsilon
\leq e^{u_\epsilon(r)}
\leq C,
\qquad 0\leq r\leq R.
\]
\end{lemma}

\begin{proof}
By \eqref{eq:fu-ode} and $u_\epsilon'(0)=0$, we have
\[
r u_\epsilon'(r)
=\frac{4\pi^2}{\epsilon^2}\int_0^r
 t\left(e^{2u_\epsilon(t)}-t^2e^{-2u_\epsilon(t)}\right)\,dt
\leq\frac{2\pi^2}{\epsilon^2}r^2e^{2u_\epsilon(r)},
\]
where the last inequality follows from the monotonicity of $u_\epsilon$.
Thus
\[
e^{-2u_\epsilon(r)}u_\epsilon'(r)\leq
2\pi^2\epsilon^{-2}r.
\]
Integrating the above over $[0,R]$ gives
\[
\frac12\left(e^{-2u_\epsilon(0)}-e^{-2u_\epsilon(R)}\right)
=\int_0^R e^{-2u_\epsilon(r)}u_\epsilon'(r)\,dr
\leq\frac{\pi^2R^2}{\epsilon^2}.
\]
Noting that $e^{-2u_\epsilon(R)}=R^{-1}$, we obtain
\[e^{-2u_\epsilon(0)}\leq R^{-1}+2\pi^2R^2\epsilon^{-2}
\leq C\epsilon^{-2},\]
 Hence $e^{u_\epsilon(0)}\geq C^{-1}\epsilon.$
The monotonicity of $u_\epsilon$ and the boundary condition then yield
\[
C^{-1}\epsilon\leq e^{u_\epsilon(0)}\leq e^{u_\epsilon(r)}
\leq e^{u_\epsilon(R)}=R^{1/2}\leq C.
\]
\end{proof}

Fix a sufficiently small $\epsilon_1>0$ and set
$H_*:=H_{0,\epsilon_1}$. Let
\[
\omega:=\omega_B+\omega_T
\]
denote the fixed standard product K\"ahler metric on $X=B\times T$.
We write
$dV:=\omega^2/2=\omega_\epsilon^2/2$ and then
$\int_XdV=1$.
By the above estimate in Lemma \ref{lem:u-bound-cn} for $u_{\epsilon}$, we can compare $H_{0,\epsilon}$ with the fixed metric $H_{*}$.
\begin{proposition}
\label{prop:h0-comparison-cn}
There exists a constant $C>1$, independent of $\epsilon$, such that for
all sufficiently small $\epsilon>0$,
\begin{equation}
C^{-1}\epsilon H_*
\leq H_{0,\epsilon}
\leq C\epsilon^{-1}H_*.
\label{eq:fu-h0-comparison-cn}
\end{equation}
In particular, for any $\alpha\in A^1(\End(V))$,
\begin{equation}
\|\alpha\|_{L^2(H_*,\omega)}^2
\leq
C\epsilon^{-3}
\|\alpha\|_{L^2(H_{0,\epsilon},\omega_\epsilon)}^2.
\label{eq:normcomp}
\end{equation}
\end{proposition}

\begin{proof}
Fix a branch point $\xi_a$, let $z_a$ be the local coordinate centered
at $\xi_a$ chosen in \cite[Section~2]{Fu}, and write $r=|z_a|$.
Recall the cut-off function $\rho$ and the functions
$\phi_1,\phi_2$ introduced in the construction of $H_{0,\epsilon}$.

Let $(\hat\mu_1^a,\hat\mu_2^a)$ be the smooth frame of
$V|_{D_{2r_0}\times T}$ chosen in \cite[Section~2]{Fu}. By
\cite[equations~(5.4)--(5.5) and the conformal normalization preceding
equation~(5.6)]{Fu}, the matrix of $H_{0,\epsilon}$ in this frame is
\begin{equation}
e^{g_a(z_a)}
\begin{pmatrix}
e^{(\phi_1-\phi_2)/2}&0\\
0&e^{(\phi_2-\phi_1)/2}
\end{pmatrix}.
\label{eq:local-h0-cn}
\end{equation}
Here $g_a$ is the real-valued harmonic function appearing in
\cite[equation~(5.1)]{Fu} and is independent of $\epsilon$.
Consequently, on the fixed disk,
$c\leq e^{g_a(z_a)}\leq C$ for constants $c,C>0$ independent of
$\epsilon$.

For $0\leq r\leq r_0$, one has
$e^{\phi_1}=e^{-u_\epsilon}$ and
$e^{\phi_2}=e^{u_\epsilon}$. On the transition region
$r_0\leq r\leq3r_0/2$, the functions $r^{1/2}$ and $r^{-1/2}$ have
uniform positive upper and lower bounds. Lemma~\ref{lem:u-bound-cn}
therefore yields
\begin{align}
c\leq e^{\phi_1}\leq C\epsilon^{-1},
\qquad&
c\epsilon\leq e^{\phi_2}\leq C,\notag\\
c\leq e^{(\phi_1-\phi_2)/2}\leq C\epsilon^{-1},
\qquad&
c\epsilon\leq e^{(\phi_2-\phi_1)/2}\leq C
\label{eq:diagonal-bounds-cn}
\end{align}
on $0\leq r\leq3r_0/2$.

Since $H_*$ is a fixed smooth Hermitian metric, the matrices of $H_*$
and $H_*^{-1}$ are uniformly bounded in the finitely many fixed smooth
frames under consideration. It follows from
\eqref{eq:local-h0-cn} and \eqref{eq:diagonal-bounds-cn} that
$c\epsilon H_*\leq H_{0,\epsilon}\leq C\epsilon^{-1}H_*$ on
$\{|z_a|\leq3r_0/2\}\times T$. Since there are only finitely many
branch points, the same estimate holds uniformly on all branch disks.

Outside the branch disks, the cut-off function vanishes and
$H_{0,\epsilon}=h_0$. Since $H_*=H_{0,\epsilon_1}=h_0$ there,
\eqref{eq:fu-h0-comparison-cn} follows on all of $X$.

It remains to prove \eqref{eq:normcomp}. At any $x\in X$, choose an
$H_*$-unitary frame in which
$H_{0,\epsilon}=\operatorname{diag}(d_1,d_2)$. By
\eqref{eq:fu-h0-comparison-cn},
$C^{-1}\epsilon\leq d_i\leq C\epsilon^{-1}$ for $i=1,2$. If $A=(A^i{}_j)\in\End(V_x)$, then
\[
|A|_{H_{0,\epsilon}}^2
=\sum_{i,j=1}^2\frac{d_i}{d_j}|A^i{}_j|^2
\geq c\epsilon^2|A|_{H_*}^2.
\]
Moreover, any one-form $\eta$ satisfies
$|\eta|_\omega^2\leq C\epsilon^{-1}|\eta|_{\omega_\epsilon}^2$.
Thus, for any
$\alpha\in A^1(\End(V))$, 
\[
|\alpha|_{H_*,\omega}^2\leq C\epsilon^{-3}
|\alpha|_{H_{0,\epsilon},\omega_\epsilon}^2.
\]
 Since the volume forms agree, integration
proves \eqref{eq:normcomp}.
\end{proof}

By Theorem~\ref{thm:approximate-mean-curvature}, after choosing
$C_0>0$ sufficiently large, there exists $c_0>0$, independent of
$\epsilon$, such that $\delta_\epsilon:=C_0e^{-c_0/\epsilon}$ bounds in
absolute value all the eigenvalues
of the endomorphism
$K_\epsilon:=\frac{\sqrt{-1}}2
\Lambda_{\omega_\epsilon}\Theta(H_{0,\epsilon})$.

We also recall the degenerating Sobolev inequality from
\cite[Lemma~16]{Fu}.

\begin{lemma}
\label{sobolev}
There exists a constant $C>0$, independent of $\epsilon$, such that
any smooth function $v$ satisfies
\begin{equation}
\|v\|_{L^4(dV)}^2
\leq
\|v\|_{L^2(dV)}^2
+
C\epsilon^{-10}
\int_X|dv|_{\omega_\epsilon}^2\,dV.
\label{eq:fu-sobolev-upper-cn}
\end{equation}
\end{lemma}

Since $H_{1,\epsilon}$ are Hermitian Yang--Mills metrics, by \cite{UY}, we have the following lemma. We use the convention
$\frac{\sqrt{-1}}2\Lambda_{\omega_\epsilon}\partial\bar\partial f
=\frac14\Delta_\epsilon f$.
\begin{lemma}
	\label{lem:fu-trace-cn}
	On $X$, one has
	\begin{equation}
		\Delta_\epsilon\Tr H_\epsilon
		\geq
		-4\delta_\epsilon\Tr H_\epsilon,
		\label{eq:trace-subsolution-cn}
	\end{equation}
	and
	\begin{equation}
		\int_X
		|\bar\partial H_\epsilon|_{H_{0,\epsilon},\omega_\epsilon}^2\,dV
		\leq
		2\delta_\epsilon
		\left(\sup_X\Tr H_\epsilon\right)
		\int_X\Tr H_\epsilon\,dV.
		\label{eq:integrated-energy-cn}
	\end{equation}
\end{lemma}

\begin{proof}
	Here $\partial_{H_{0,\epsilon}}H_\epsilon$ denotes the $(1,0)$ part of
	the connection induced on $\End(V)$ by the Chern connection of
	$H_{0,\epsilon}$, applied to $H_\epsilon$; see
	\cite[Section~1.9]{Siu1987}.
	
	The standard Laplacian identity for the trace of the relative
	endomorphism, equivalently the invariant form of
	\cite[equation~(6.4)]{Fu}, together with the matrix conventions in
	\cite[equations~(3.8)--(3.9)]{Fu},
	$\Lambda_{\omega_\epsilon}\Theta(H_{1,\epsilon})=0$, and the convention
	$\frac{\sqrt{-1}}2\Lambda_{\omega_\epsilon}\partial\bar\partial f
	=\frac14\Delta_\epsilon f$, gives
	\begin{equation}
		\frac14\Delta_\epsilon\Tr H_\epsilon
		=
		Q_\epsilon+\Tr(K_\epsilon H_\epsilon),
		\qquad
		Q_\epsilon
		:=
		-\frac{\sqrt{-1}}2
		\Lambda_{\omega_\epsilon}
		\Tr\left(
		\bar\partial H_\epsilon H_\epsilon^{-1}
		\wedge
		\partial_{H_{0,\epsilon}}H_\epsilon
		\right).
		\label{eq:fu-trace-cn}
	\end{equation}
	Both $Q_\epsilon$ and $\Tr(K_\epsilon H_\epsilon)$ are intrinsically
	defined smooth real-valued functions on $X$.
	
	Fix $x_0\in X$ and choose an $H_{0,\epsilon}$-normal local holomorphic
	frame near $x_0$. Thus
	$\widetilde H_{0,\epsilon}(x_0)=I$ and
	$\partial\widetilde H_{0,\epsilon}(x_0)=0$. Since $H_\epsilon$ is
	self-adjoint with respect to $H_{0,\epsilon}$, at $x_0$ one has
	$\widetilde H_{\epsilon,\bar z}
	=\widetilde H_{\epsilon,z}^*$ and
	$\widetilde H_{\epsilon,\bar w}
	=\widetilde H_{\epsilon,w}^*$.
	
	Since $H_\epsilon$ is positive definite, at $x_0$ one also has
	\[
	\widetilde H_\epsilon(x_0)
	\leq
	\Tr H_\epsilon(x_0)I
	\leq
	\left(\sup_X\Tr H_\epsilon\right)I,
	\]
	and hence
	\begin{align*}
		Q_\epsilon
		&=
		\epsilon\Tr\left(
		\widetilde H_{\epsilon,z}^*
		\widetilde H_\epsilon^{-1}
		\widetilde H_{\epsilon,z}
		\right)
		+
		\epsilon^{-1}\Tr\left(
		\widetilde H_{\epsilon,w}^*
		\widetilde H_\epsilon^{-1}
		\widetilde H_{\epsilon,w}
		\right)\\
		&\geq
		\left(\sup_X\Tr H_\epsilon\right)^{-1}
		\left[
		\epsilon\Tr\left(
		\widetilde H_{\epsilon,z}^*
		\widetilde H_{\epsilon,z}
		\right)
		+
		\epsilon^{-1}\Tr\left(
		\widetilde H_{\epsilon,w}^*
		\widetilde H_{\epsilon,w}
		\right)
		\right]\\
		&=
		\frac{1}{2}
		\left(\sup_X\Tr H_\epsilon\right)^{-1}
		|\bar\partial H_\epsilon|_{H_{0,\epsilon},\omega_\epsilon}^2.
	\end{align*}
	Thus
	\begin{equation}
		|\bar\partial H_\epsilon|_{H_{0,\epsilon},\omega_\epsilon}^2
		\leq
		2\bigl(\sup_X\Tr H_\epsilon\bigr)Q_\epsilon.
		\label{eq:weighted-first-order-cn}
	\end{equation}
	
	 By the definition of $\delta_\epsilon$,
	$-\delta_\epsilon\Id\leq K_\epsilon\leq\delta_\epsilon\Id$. Since
	$H_\epsilon$ is positive definite, we have
	\[
	|\Tr(K_\epsilon H_\epsilon)|\leq\delta_\epsilon\Tr H_\epsilon.
	\]
	Combining this estimate with \eqref{eq:fu-trace-cn} and
	$Q_\epsilon\geq0$ proves \eqref{eq:trace-subsolution-cn}.
	
	Finally, integration of \eqref{eq:fu-trace-cn} over $X$ gives
	\[
	\int_XQ_\epsilon\,dV
	=
	-\int_X\Tr(K_\epsilon H_\epsilon)\,dV
	\leq
	\delta_\epsilon\int_X\Tr H_\epsilon\,dV.
	\]
	Together with
	\eqref{eq:weighted-first-order-cn}, this proves
	\eqref{eq:integrated-energy-cn}.
\end{proof}

Since $\delta_{\epsilon}=C_0e^{-\frac{c_0
	}{\epsilon}}$, we can control the $L^\infty$ norm of $\Tr H_\epsilon$ by its
$L^1$ norm in the following stronger form.

\begin{lemma}
\label{lem:L1-Linfty-cn}
There exist constants $c>0$ and $\epsilon_0>0$, independent of
$\epsilon$, such that for any $0<\epsilon\leq\epsilon_0$,
\begin{equation}
\|\Tr H_\epsilon\|_{L^\infty}
\leq
\left(1+e^{-c/\epsilon}\right)
\|\Tr H_\epsilon\|_{L^1(dV)}.
\label{eq:L1-Linfty-cn}
\end{equation}
\end{lemma}

\begin{proof}
Set $u=\Tr H_\epsilon>0$. By
\eqref{eq:trace-subsolution-cn},
\begin{equation}
-\Delta_\epsilon u
\leq
4\delta_\epsilon u.
\label{eq:u-subsolution-cn}
\end{equation}
For $p\geq1$, multiply \eqref{eq:u-subsolution-cn} by $u^{2p-1}$ and
integrate by parts. This gives
\[
\frac{2p-1}{p^2}
\int_X|d(u^p)|_{\omega_\epsilon}^2\,dV
\leq
4\delta_\epsilon
\int_Xu^{2p}\,dV.
\]
Since $p^2/(2p-1)\leq p$,
\begin{equation}
\int_X|d(u^p)|_{\omega_\epsilon}^2\,dV
\leq
4p\delta_\epsilon
\int_Xu^{2p}\,dV.
\label{eq:moser-gradient-cn}
\end{equation}
Applying \eqref{eq:fu-sobolev-upper-cn} to $v=u^p$ and using
\eqref{eq:moser-gradient-cn}, we obtain
\begin{equation}
\|u\|_{L^{4p}}
\leq
\left(
1+Cp\epsilon^{-10}\delta_\epsilon
\right)^{1/(2p)}
\|u\|_{L^{2p}}.
\label{eq:moser-step-cn}
\end{equation}

Set $q_\epsilon=C\epsilon^{-10}\delta_\epsilon$, where $C$ is the
constant in \eqref{eq:moser-step-cn}. Taking $p=2^j$ and iterating gives
\[
\|u\|_{L^\infty}
\leq
\prod_{j=0}^{\infty}
\left(
1+2^jq_\epsilon
\right)^{1/2^{j+1}}
\|u\|_{L^2}.
\]
Since $\log(1+s)\leq C\sqrt{s}$ for $s\geq0$,
\[
\log
\prod_{j=0}^{\infty}
\left(
1+2^jq_\epsilon
\right)^{1/2^{j+1}}
\leq
C\sqrt{q_\epsilon}
\sum_{j=0}^{\infty}\frac{1}{2^{\frac{j}{2}}}
\leq
C\sqrt{q_\epsilon}.
\]
Since
$\sqrt{q_\epsilon}\leq
C\epsilon^{-5}e^{-c_0/(2\epsilon)}$, after decreasing $c$ and
$\epsilon_0$ the preceding product is at most
$1+e^{-c/\epsilon}$. Hence
\begin{equation}
\|u\|_{L^\infty}
\leq
\left(
1+e^{-c/\epsilon}
\right)\|u\|_{L^2}.
\label{eq:L2-Linfty-cn}
\end{equation}

Taking $p=1$ in \eqref{eq:moser-step-cn} gives
\[
\|u\|_{L^4}\leq(1+q_\epsilon)^{1/2}\|u\|_{L^2}.
\]
Then by the interpolation
inequality
\[
\|u\|_{L^2}\leq\|u\|_{L^1}^{1/3}\|u\|_{L^4}^{2/3},
\]
we obtain
\[
\|u\|_{L^2}\leq(1+q_\epsilon)\|u\|_{L^1}.
\]
Substituting this into
\eqref{eq:L2-Linfty-cn}, using
$q_\epsilon\leq e^{-c/\epsilon}$ after another decrease of $c$ and
$\epsilon_0$, and then decreasing them once more,
proves \eqref{eq:L1-Linfty-cn}.
\end{proof}

With respect to the fixed metrics $(H_*,\omega)$, for
$A,B\in A^0(\End(V))$, define
\[
\langle A,B\rangle_{L^2(H_*,\omega)}
:=\int_X\Tr(AB^{*H_*})\,dV.
\]
In particular,
$\|B\|_{L^2(H_*,\omega)}^2=\int_X\Tr(BB^{*H_*})\,dV$.

We also need the following Poincar\'e type inequality which is standard. For completeness, we include a proof.

\begin{lemma}
	\label{lem:dolbeault-poincare-cn}
	Let $\bar\partial^*$ be the formal adjoint of $\bar\partial$ with respect
	to the $L^2$ inner product induced by $(H_*,\omega)$. There exists a
	constant $C>0$, depending only on $H_*$, $\omega$, and $V$, such that
	 for any $B\in A^0(\End(V))$ satisfying
	\[
	\int_X\Tr B\,dV=0,
	\]
we have
	\[
	\|B\|_{L^2(H_*,\omega)}^2
	\leq
	C\|\bar\partial B\|_{L^2(H_*,\omega)}^2.
	\]
\end{lemma}

\begin{proof}
	As recalled in Section~\ref{sec:preliminaries}, the adiabatic argument
	of \cite{FMW} shows that $V$ is stable with respect to
	$\omega_\epsilon$ for all sufficiently small $\epsilon$. A stable
	holomorphic vector bundle is simple by
	\cite[Chapter~V, Corollary~7.14]{Kob}; hence
	\[
	H^0(X,\End(V))=\mathbb C\Id.
	\]
	Moreover,
	\[
	\ker(\bar\partial^*\bar\partial)
	=
	\ker\bar\partial
	=
	H^0(X,\End(V))
	=
	\mathbb C\Id.
	\]
	
	Suppose, for contradiction, that the asserted inequality does not
	hold. Then there exists a sequence
	$B_j\in A^0(\End(V))$ such that
	\[
	\int_X\Tr B_j\,dV=0,
	\qquad
	\|B_j\|_{L^2(H_*,\omega)}=1,
	\qquad
	\|\bar\partial B_j\|_{L^2(H_*,\omega)}\longrightarrow0.
	\]
	Since
	\[
	\langle B_j,\Id\rangle_{L^2(H_*,\omega)}
	=
	\int_X\Tr B_j\,dV=0,
	\]
	each $B_j$ is orthogonal to
	$\ker(\bar\partial^*\bar\partial)=\mathbb C\Id$.
	
	Consider the elliptic operator
	\[
	P=\bar\partial+\bar\partial^*
	\]
	acting on $\End(V)$-valued $(0,*)$-forms. By the standard elliptic
	estimate \cite[Theorem~1.4.1]{Joyce},
	\[
	\|B_j\|_{W^{1,2}(H_*,\omega)}
	\leq
	C\left(
	\|PB_j\|_{L^2(H_*,\omega)}
	+
	\|B_j\|_{L^2(H_*,\omega)}
	\right).
	\]
	Since $B_j$ is an $\End(V)$-valued $(0,0)$-form, one has
	$\bar\partial^*B_j=0$ and then
	\[
	PB_j=\bar\partial B_j.
	\]
	Thus $\{B_j\}$ is uniformly bounded in
	$W^{1,2}(X,\End(V))$.
	
	By the Rellich--Kondrachov theorem
	\cite[Chapter~2, Section~11]{Aubin}, after passing to a subsequence,
	there exists $B_\infty\in L^2(X,\End(V))$ such that
	\[
	B_j\longrightarrow B_\infty
	\qquad\text{strongly in }L^2(X,\End(V)).
	\]
	In particular,
	\[
	\|B_\infty\|_{L^2(H_*,\omega)}=1.
	\]
	
	For any smooth $\End(V)$-valued $(0,1)$-form $\eta$, we have
	\begin{align*}
		\left\langle B_\infty,\bar\partial^*\eta
		\right\rangle_{L^2(H_*,\omega)}
		&=
		\lim_{j\to\infty}
		\left\langle B_j,\bar\partial^*\eta
		\right\rangle_{L^2(H_*,\omega)}
		\\
		&=
		\lim_{j\to\infty}
		\left\langle\bar\partial B_j,\eta
		\right\rangle_{L^2(H_*,\omega)}
		=0.
	\end{align*}
	Thus $\bar\partial B_\infty=0$ in the distributional sense. Since
	$P=\bar\partial+\bar\partial^*$ is elliptic, elliptic regularity
	implies that $B_\infty$ is smooth. Therefore
	\[
	B_\infty\in
	\ker\bar\partial
	=
	H^0(X,\End(V))
	=
	\mathbb C\Id.
	\]
	
	On the other hand,  $B_j\perp \mathbb C\Id$ and the $L^2$-convergence of $\{B_j\}$ implies
	$
	B_\infty\perp\mathbb C\Id.
	$
	Thus $B_\infty=0$, contradicting
	$\|B_\infty\|_{L^2(H_*,\omega)}=1$. This proves the desired
	inequality.
	
	Since $H_*$, $\omega$, and $V$ are fixed, the constant $C$
	is independent of $\epsilon$.
\end{proof}

We next prove the exponential $C^0$ decay estimate
\begin{proposition}
\label{prop:trace-exponential-cn}
There exist constants $c>0$ and $\epsilon_0>0$, independent of
$\epsilon$, such that for any $0<\epsilon\leq\epsilon_0$,
\begin{equation}
\sup_X\left|\Tr H_\epsilon-2\right|
\leq
e^{-c/\epsilon}.
\label{eq:trace-exponential-cn}
\end{equation}
\end{proposition}

\begin{proof}
Set
$B_\epsilon:=H_\epsilon-a_\epsilon\Id$, where
\[
a_\epsilon:=\frac12\int_X\Tr H_\epsilon\,dV.
\]
Since $\int_XdV=1$,
Lemma~\ref{lem:L1-Linfty-cn} gives
\begin{align}
\|\Tr H_\epsilon\|_{L^\infty}
\notag\leq&
\left(
1+e^{-c/\epsilon}
\right)
\int_X\Tr H_\epsilon\,dV
\\=&
2\left(
1+e^{-c/\epsilon}
\right)a_\epsilon.
\label{eq:M-by-a-cn}
\end{align}

By \eqref{eq:integrated-energy-cn} and
Proposition~\ref{prop:h0-comparison-cn},
\[
\|\bar\partial H_\epsilon\|_{L^2(H_*,\omega)}^2
\leq
C\epsilon^{-3}\delta_\epsilon
\|\Tr H_\epsilon\|_{L^\infty(X)}
\int_X\Tr H_\epsilon\,dV.
\]
Combining this with \eqref{eq:M-by-a-cn} and the definition of
$\delta_\epsilon$, and absorbing the factor
$\epsilon^{-3}$ into the exponential, we obtain, for some $c_1>0$,
\begin{equation}
\|\bar\partial H_\epsilon\|_{L^2(H_*,\omega)}^2
\leq
Ca_\epsilon^2e^{-c_1/\epsilon}.
\label{eq:H-dbar-exponential-cn}
\end{equation}

By definition,
$\int_X\Tr B_\epsilon\,dV=0$ and
$\bar\partial B_\epsilon=\bar\partial H_\epsilon$. Hence
Lemma~\ref{lem:dolbeault-poincare-cn} and
\eqref{eq:H-dbar-exponential-cn} imply
\begin{equation}
\|B_\epsilon\|_{L^2(H_*,\omega)}^2
\leq
Ca_\epsilon^2e^{-c_1/\epsilon}.
\label{eq:B-exponential-cn}
\end{equation}

Since $H_\epsilon$ is positive definite and $\det H_\epsilon=1$, one
has $\Tr H_\epsilon\geq2$ pointwise, and hence $a_\epsilon\geq1$.
Since $V$ has rank two,
\[
1=\det H_\epsilon
=\det(a_\epsilon\Id+B_\epsilon)
=a_\epsilon^2+a_\epsilon\Tr B_\epsilon+\det B_\epsilon.
\]
Integrating this identity and using $\int_XdV=1$ and
$\int_X\Tr B_\epsilon\,dV=0$, we obtain
\[
a_\epsilon^2-1=-\int_X\det B_\epsilon\,dV.
\]
Since
\[
|\det B_\epsilon|
\leq\frac12|B_\epsilon|_{H_*}^2,
\]
we obtain
\begin{equation}
0\leq
a_\epsilon^2-1
\leq
\left|\int_X\det B_\epsilon\,dV\right|
\leq
\frac12
\|B_\epsilon\|_{L^2(H_*,\omega)}^2.
\label{eq:det-average-cn}
\end{equation}

It follows from \eqref{eq:B-exponential-cn} and
\eqref{eq:det-average-cn} that
\begin{equation}
0\leq a_\epsilon^2-1
\leq
Ca_\epsilon^2e^{-c_1/\epsilon}.
\label{eq:a-square-cn}
\end{equation}
After decreasing $\epsilon_0$, assume
$Ce^{-c_1/\epsilon}\leq\frac12$. Then
\eqref{eq:a-square-cn} first gives $a_\epsilon^2\leq2$, and hence,
after decreasing the exponential rate if necessary, we get
\begin{equation}
0\leq a_\epsilon-1
\leq
e^{-c_2/\epsilon},
\label{eq:a-close-one-cn}
\end{equation}
for some $c_2>0$.

Therefore, Lemma~\ref{lem:L1-Linfty-cn} and
\eqref{eq:a-close-one-cn} yield
\[
\|\Tr H_\epsilon\|_{L^{\infty}(X)}\leq2(1+e^{-c/\epsilon})a_\epsilon
\leq2+Ce^{-c_3/\epsilon},
\]
for some $c_3>0$. Since $\Tr H_\epsilon\geq2$ pointwise, the
 constant $C$ can be absorbed by decreasing $c_3$ and
$\epsilon_0$. This proves \eqref{eq:trace-exponential-cn}.
\end{proof}

\begin{proof}[Proof of the $C^0$ estimate in
Theorem~\ref{Thm1.1}]
Fix $x\in X$, and let $\lambda_\epsilon(x)\geq1$ be the larger
eigenvalue of the $H_{0,\epsilon}$-self-adjoint endomorphism
$H_\epsilon(x)$. Since $\det H_\epsilon=1$, the other eigenvalue is
$\lambda_\epsilon(x)^{-1}$, and
\[
\Tr H_\epsilon-2
=
\lambda_\epsilon+\lambda_\epsilon^{-1}-2
=
\frac{(\lambda_\epsilon-1)^2}{\lambda_\epsilon}.
\]
By Proposition~\ref{prop:trace-exponential-cn}, for sufficiently small
$\epsilon$ one has
$\lambda_\epsilon\leq\Tr H_\epsilon\leq3$. Therefore,
\[
|H_\epsilon-\Id|_{H_{0,\epsilon}}^2
=(\lambda_\epsilon-1)^2
+(\lambda_\epsilon^{-1}-1)^2
=(\lambda_\epsilon+\lambda_\epsilon^{-1})
(\Tr H_\epsilon-2)
\leq3e^{-c/\epsilon}.
\]
Then we obtain the desired
$C^0$ estimate by decreasing  $c$.
\end{proof}

\begin{proof}[Proof of the higher order estimates in Theorem~\ref{Thm1.1}]
We repeat the argument in the proof of \cite[Theorem~3]{Fu}, using the
exponential $C^0$ estimate above and the exponential mean-curvature
estimate \eqref{eq:mean-curvature-exponential}. For any fixed
derivative order, the coordinate rescalings and coefficient estimates
in that argument introduce only finitely many negative powers of
$\epsilon$. Thus all resulting error terms are bounded by finite sums
of terms of the form $\epsilon^{-N}e^{-c/\epsilon}$.

For any  $N>0$ and $c>0$, there exists $c'>0$ such that
\[
\epsilon^{-N}e^{-c/\epsilon}\leq e^{-c'/\epsilon},
\]
for all sufficiently small $\epsilon$. Consequently, for any
nonnegative integer $k$, there exist constants $C_k,c_k>0$, independent
of $\epsilon$, such that
\[
\|H_\epsilon-\Id\|_{C^k(H_{0,\epsilon})}
\leq C_ke^{-c_k/\epsilon}.
\]
This completes the proof of Theorem~\ref{Thm1.1}.
\end{proof}
\section*{Acknowledgments}
Fu is grateful to Professor Jun Li for his continued guidance and
encouragement. In particular, the problem studied in this paper was
suggested to him by Professor Li. Fu and Zhang also thank Professors
Kefeng Liu, Xiaokui Yang and Shing-Tung Yau, Weiping Zhang for helpful discussions.
\begingroup
\sloppy

\endgroup
\bigskip

\noindent
\textsc{Jixiang Fu}\\
Shanghai Center for Mathematical Sciences, Fudan University,
Shanghai 200433, China\\
\textit{E-mail address:}
\texttt{majxfu@fudan.edu.cn}

\medskip

\noindent
\textsc{Dekai Zhang}\\
School of Mathematical Sciences, Key Laboratory of Mathematics and
Engineering Applications (Ministry of Education), Shanghai Key
Laboratory of PMMP, East China Normal University,
Shanghai 200241, China\\
\textit{E-mail address:}
\texttt{dkzhang@math.ecnu.edu.cn}

\begin{thebibliography}{99}

\bibitem{Aubin}
T. Aubin,
\newblock \emph{Some Nonlinear Problems in Riemannian Geometry},
\newblock Springer Monographs in Mathematics,
Springer-Verlag, Berlin, 1998.

\bibitem{ChenViaclovskyZhang2020}
G. Chen, J. Viaclovsky and R. Zhang,
\newblock Collapsing Ricci-flat metrics on elliptic K3 surfaces,
\newblock \emph{Comm.  Anal. Geom.} \textbf{28} (2020),
2019--2133.

\bibitem{DatarJacob2022}
V. Datar and A. Jacob,
\newblock Hermitian--Yang--Mills connections on collapsing elliptically
fibered K3 surfaces,
\newblock \emph{J. Geom. Anal.} \textbf{32} (2022),
30 pp.

\bibitem{DatarJacobZhang2021}
V. Datar, A. Jacob and Y. Zhang,
\newblock Adiabatic limits of anti-self-dual connections on collapsed K3 surfaces,
\newblock \emph{J. Differential Geom.} \textbf{118} (2021),
223--296.

\bibitem{Don1985}
S. K. Donaldson,
\newblock Anti self-dual Yang--Mills connections over complex algebraic surfaces
and stable vector bundles,
\newblock \emph{Proc.  London Math. Soc.}
\textbf{50} (1985),
1--26.

\bibitem{Don1987}
S. K. Donaldson,
\newblock Infinite determinants, stable bundles and curvature,
\newblock \emph{Duke Math. J.} \textbf{54} (1987),
231--247.

\bibitem{Fri}
R. Friedman,
\newblock Rank two vector bundles over regular elliptic surfaces,
\newblock \emph{Invent. Math.} \textbf{96} (1989),
283--332.

\bibitem{FMW}
R. Friedman, J. Morgan and E. Witten,
\newblock Vector bundles and \(F\) theory,
\newblock \emph{Comm.  Math. Phys.} \textbf{187} (1997),
679--743.

\bibitem{Fu}
J. Fu,
\newblock Limiting behavior of a class of Hermitian Yang--Mills metrics, I,
\newblock \emph{Sci. China Math.} \textbf{62} (2019),
2155--2194.


\bibitem{GTZ2013}
M. Gross, V. Tosatti and Y. Zhang,
\newblock Collapsing of abelian fibred Calabi--Yau manifolds,
\newblock \emph{Duke Math. J.} \textbf{162} (2013),
517--551.

\bibitem{GW}
M. Gross and P. M. H. Wilson,
\newblock Large complex structure limits of K3 surfaces,
\newblock \emph{J. Differential Geom.} \textbf{55} (2000),
475--546.

\bibitem{HeinTosatti2020}
H.-J. Hein and V. Tosatti,
\newblock Higher-order estimates for collapsing Calabi--Yau metrics,
\newblock \emph{Camb. J.  Math.} \textbf{8} (2020),
683--773.

\bibitem{HeinTosatti2025}
H.-J. Hein and V. Tosatti,
\newblock Smooth asymptotics for collapsing Calabi--Yau metrics,
\newblock \emph{Comm.  Pure  Appl. Math.}
\textbf{78} (2025), 382--499.

\bibitem{Joyce}
D. D. Joyce,
\newblock \emph{Compact Manifolds with Special Holonomy},
\newblock Oxford Mathematical Monographs,
Oxford University Press, Oxford, 2000.

\bibitem{Kob}
S. Kobayashi,
\newblock \emph{Differential Geometry of Complex Vector Bundles},
\newblock Princeton University Press, Princeton, 1987.

\bibitem{Kon}
M. Kontsevich,
\newblock Homological algebra of mirror symmetry,
\newblock in \emph{Proceedings of the International Congress of Mathematicians,
Z\"urich, 1994}, Vols.~1--2, Birkh\"auser, Basel, 1995, 120--139.

\bibitem{Leung}
N. C. Leung,
\newblock Geometric aspects of mirror symmetry (with SYZ for rigid
CY manifolds),
\newblock in \emph{Second International Congress of Chinese Mathematicians},
New Studies in Advanced Mathematics, vol.~4,
International Press, Somerville, MA, 2004, 305--342.



\bibitem{RuanZhang2011}
W.-D. Ruan and Y. Zhang,
\newblock Convergence of Calabi--Yau manifolds,
\newblock \emph{Adv. Math.} \textbf{228} (2011),
1543--1589.

\bibitem{Siu1987}
Y.-T. Siu,
\newblock \emph{Lectures on Hermitian--Einstein Metrics for Stable Bundles
and K\"ahler--Einstein Metrics},
\newblock DMV Seminar, vol.~8, Birkh\"auser, Basel, 1987.

\bibitem{SYZ}
A. Strominger, S.-T. Yau and E. Zaslow,
\newblock Mirror symmetry is \(T\)-duality,
\newblock \emph{Nuclear Physics B} \textbf{479} (1996),
243--259.

\bibitem{Tosatti2010}
V. Tosatti,
\newblock Adiabatic limits of Ricci-flat K\"ahler metrics,
\newblock \emph{J. Differential Geom.} \textbf{84} (2010),
427--453.

\bibitem{UY}
K. Uhlenbeck and S.-T. Yau,
\newblock On the existence of Hermitian--Yang--Mills connections in stable
vector bundles,
\newblock \emph{Comm.  Pure  Appl. Math.}
\textbf{39} (1986), S257--S293.


\bibitem{Wilson2004}
P. M. H. Wilson,
\newblock Metric limits of Calabi--Yau manifolds,
\newblock in \emph{The Fano Conference},
University of Torino, Turin, 2004, 793--804.

\bibitem{Yau}
S.-T. Yau,
\newblock On the Ricci curvature of a compact K\"ahler manifold and the
complex Monge--Amp\`ere equation, I,
\newblock \emph{Comm.  Pure  Appl. Math.}
\textbf{31} (1978), 339--411.

\bibitem{Zharkov2004}
I. Zharkov,
\newblock Limiting behavior of local Calabi--Yau metrics,
\newblock \emph{Adv. Theor. Math. Phys.}
\textbf{8} (2004), 395--420.

\end{thebibliography}
\end{document}